\documentclass[12pt,a4paper]{article}
\usepackage[top=1in, bottom=1in, right=1in, left=1in, includefoot]{geometry}
\usepackage{mathtools}
\usepackage[english]{babel}
\usepackage{amsthm}

\newcommand\ov{\overline}

\newcommand\R{\mathbb R}

\newcommand\C{\mathbb C}

\newcommand\p{\partial}

\renewcommand\l{\left}
\renewcommand\r{\right}
\allowdisplaybreaks
\renewcommand\phi\varphi
\newcommand\eq[2]{\begin{equation}#1\label{#2}\end{equation}}

\newcommand\Ra\Rightarrow
\renewcommand\a\alpha
\renewcommand\b\beta
\renewcommand\t[1]{\smash{\tilde{#1}}\vphantom{#1}}
\newcommand\h[1]{\smash{\hat{#1}}\vphantom{\rule{1pt}{11pt}{#1}}}

\usepackage[dvipsnames]{xcolor}

\usepackage{fancyhdr}
\newtheorem{theorem}{Theorem}
\newtheorem{lemma}{Lemma}

\usepackage{biblatex}
\newcommand\D{\mathbb D}
\newcommand\T{\mathbb T}

\newcommand\I{\mathcal I}

\usepackage{stix2}

\usepackage[integrals]{wasysym}

\begin{document}
    \title{Proof of the generalized Chui conjecture}
    \thispagestyle{empty}
    \ 
    \begin{center}
        \uppercase{Proof of the generalized Chui conjecture}

        \ 

        Egor Bezrukov
    \end{center}
    \ 
    \begin{abstract}
        \noindent We prove Chui's conjecture on sums of Cauchy kernels. It states that among all configurations of Cauchy kernels on the unit disk, the equispaced one minimizes $L^1$ norm of the sum in the unit disk. Moreover, we obtain a generalization to  a wider set of functionals, including weighted $L^p$ norms for $p > 1$ with radially symmetric weights.
    \end{abstract}

    \section{Introduction}
    The Chui's conjecture, introduced in 1971 \cite{chui1971lower}, in its physical interpretation, states that among all distributions of $n$ unit electrical charges on the unit circle of a 2d-plane, the equispaced one (where the charges are at the vertices of a regular n-gon) minimizes $L^1$ norm of electrical field strength in the unit disk. Note that on a 2d-plane, electrical field force is proportional to $r^{-1}$, where $r$ is the distance between charges, unlike in 3d-case where it's proportional to $r^{-2}$. This is dictated by the harmonicity of the electrical potential. In the complex analysis the statement of  the conjecture is naturally formulated in terms of $L^1$-norm on the unit disk and Cauchy kernel sums as follows.
    Define the functional
    \eq{I_n(a) = \int_{|z| \leq 1} \l| \sum_{k=1}^n \frac1{z - a_k} \r| dA,}{chui}
    where $a=(a_1,\dots, a_n)\in\mathbb C^n$, $|a_k| = 1$, $k=1,\dots,n$, $ dA$ denotes an area measure. Then, for any integer $n \geq 1$, $I_n$ attains its minimum at the equispaced arrangement of poles:
    \eq{a_k = \exp(2 k \pi i / n),\  k=1,\dots,n .}{conda}

    \textbf{Known results.} Next year after Chui's publication, D. J. Newman \cite{newman1972lower} obtained a lower bound  $L^1$-norms:
    \[\l\Vert \sum_{k=1}^n \frac1{z - a_k} \r\Vert_{L^1(\D)} \geq \frac\pi{18},\]
    where $|a_k| = 1,$ $k=1,\dots, n$.
    
    In 2020, E. Abakumov, A. Borichev and K. Fedorovskiy \cite{abakumov2021chui} proved a variation 
    of Chui's conjecture for weighted $L^2$-norms. Let $g : [0, 1] \to \R, g \not\equiv 0, g(0) = 0$ be a non-decreasing function such that $\int_0^1 \frac{g(s)}s ds < \infty$. Then,
    \[\int_\D \l|\sum_{k=1}^n \frac1{z - a_k}\r|^2 g(1 - |z|^2) dA\]
    attains its minimum if and only if $\{a_k\}$ are equispaced.

    In this paper, we prove the Chui's conjecture. In fact, we establish the following more general result.
    \begin{samepage}
    \begin{theorem}
        \label{tgen}
        Let $\rho : [0, 1] \to [0, +\infty)$ be a measurable positive weight such that
        \[\int_0^1 \rho(r) r dr < \infty ,\]
        and $\Psi : \R \to \R$ be an even convex function. Set
        \eq{\I_n(a,b) = \int_{|z| \leq 1} \rho(|z|) \Psi\l(\l| \frac{\prod_{k=1}^{n-1} (z - b_k)}{\prod_{k=1}^n (z - a_k)} \r|\r) dA}{gen}
        where $a=(a_1,\dots ,a_n)\in\mathbb C^n,$ and $b=(b_1, \dots , b_{n-1})\in\mathbb C^{n-1}$  with $|a_k| = 1$ and no restriction on $b_k$. Then,
         for any integer $n \geq 1$, $\I_n$ attains its minimum at
        \eq{\begin{aligned}
            &a_k = \exp(2 k \pi i / n), &&k=1,\dots, n,\\
            &b_j = 0,                   &&j=1,\dots, n-1.
        \end{aligned}}{condb}
    \end{theorem}
    \end{samepage}
    Note that, without an additional condition on $\Psi$, $\I_n$ may equal $+\infty$.

    The original Chui conjecture is a particular case of the assertion of Theorem 1. Indeed, reduce the Cauchy kernel sum to a common denominator:
    \eq{\sum_{k=1}^n \frac1{z - a_k} = \frac{n p_{n-1}(z)}{\prod_{k=1}^n (z - a_k)} = n \frac{\prod_{k=1}^{n-1} (z - b_k)}{\prod_{k=1}^n (z - a_k)},}{cdm}
    for some monic polynomial $p_{n-1}$ and its roots $b_k, k = 1 .. n-1$, depending on the choice of $a_k$. 
    For $a_k = \exp(2 k \pi i / n)$ we get
    \[\sum_{k=1}^n \frac1{z - a_k} = \frac{\sum_{k=1}^n \prod_{j=1}^{n-1} (z - \exp(2 (k + j) \pi i / n))}{\prod_{k=1}^n (z - \exp(2 k \pi i / n))} = n \frac{z^{n-1}}{\prod_{k=1}^n (z - a_k)},\]
   that  corresponds to $b_k = 0$, $k=1,\dots, n-1.$

    Now it is easy to see that after setting $\rho(r) = 1/n$ and $\Psi(x) = |x|$, (\ref{gen}) becomes the expanded form of (\ref{chui}) with condition (\ref{condb}) corresponding to condition (\ref{conda}). So if $\I_{n}$ attains its minimum at (\ref{condb}) then $I_n$ must also attain its minimum at (\ref{conda}) because it's the same functional on a restricted (by the $a_k \to b_k$ mapping) domain.

    Setting $\Psi(x) = |x|^p$, $p \geq 1$ we also have a variation of Chui's conjecture for $L^p$-norm, that contains  the case proved in [1].

    The key role will play that  for an even convex function $\Phi$ and a quotient  of monic polynomials $P$ having $n$ fixed (!) poles $\{\a_k\}$ outside the unit disk and
    $n - 1$ arbitrary zeros $\{\b_k\}$, the minimum of $\int_{|z|=1} \Phi(|P(z)|) |dz|$ over $(\b_k) \in \C^{n-1}$ does not depend on the particular configuration of $\a_k$
    but only depends on $\prod |\a_k|$.

    \section{Proof overview}

    First we reformulate the conjecture to a more general form as already shown above. Reduction to common denominator is a natural step because the modulus of a sum in the original conjecture has little to no good properties, whereas modulus of a product is much better as the modulus is a multiplicative function.

    We integrate $\Psi(...)$ over circles $|z| = r, r \in (0, 1)$. $\I_n$ is the weighted integral over $r \in (0, 1)$ of such integrals.
    It turns out that $\int_{|z|=r} \Psi(...) |dz|$ has a fascinating property that has been discovered during numerical experiments: if we fix the values of $a_k$ and minimize 
    the integral by varying $b_k$ only, {\it the minimum value will not depend on the initial choice of $a_k$.}
     We show why it happens and what  is the optimal choice of $b_k$ for given values $a_k$.
     Further we prove that exactly the equispaced arrangement of $a_k$ makes it so that the single choice $b_k = 0,$ $k=0,\dots, n-1, $ minimizes all circular integrals \textit{at the same time}. Note that other arrangements of $a_k$ may suggest different optimal arrangements of $b_k$ for different circles.

    \section{Minimization on cirles} \label{circles1}

    In this section, we obtain the central result necessary for the proof of Theorem \ref{tgen}.

    \textbf{Terminology.} As usual, $\T = \{z | z \in \C, |z| = 1\}$. For an inner function in the unit disk $g$, analytic on $\ov\D$, define its \textit{degree} $m$ to be such natural number that under the smooth map $\Theta \to \arg g(e^{i\Theta})$ one has $\arg g(e^{i(\Theta + 2\pi)}) = \arg g(e^{i\Theta}) + 2\pi m$. Note that the degree of a finite Blaschke product equals the number of terms in the product. Moreover, any inner function $f$ in the unit disk, analytic on $\ov\D$, is a finite Blaschke product up to a unimodular constant factor. Indeed, $f$ may only have a finite number of zeros in $\D$. Thus, after dividing $f$ by the Blaschke product with the same zeros we get a function $g$ which modulus equals $1$ on $\T$ and it has no zeros in $\D$. It imlpies that $g$ must equal a unimodular constant.

    \begin{lemma}
        \label{Clema}
        Let
        \[T(z) = \prod_{j=1}^n \frac{z - s_j}{1 - \ov{s_j} z},\]
        $|s_j| < 1, j = 1 .. n$ be a Blaschke product (up to a unimodular constant). Then for any function $f : \T \to \C, f \in L^1(\T)$, all numbers
        \[C_k = \int_\T \frac{f(T(z))}{1 - \ov{s_k} z} |dz|\]
        are pairwise equal, i.e. $\exists C : C_k = C\ \forall k = 1..n$.
    \end{lemma}
    \begin{proof}
        Note that the degree of $T$ is $n$ and the map $\Theta \mapsto \arg(T(e^{i\Theta}))$ can be expressed as a smooth strictly increasing function.
        Thus we can split $\T$ into precisely $n$ arcs $U_j = \{e^{i\Theta} | \Theta \in [\Theta_{j-1}, \Theta_j)\}$ such that $T(U_j) = \T$ and $T$ is invertible on any of $U_j$. Let the inverse of $T$ on $U_j$ be $r_j : \T \to U_j$.

        Using this we can rewrite the definition for $C_k$:
        \begin{samepage}
            \[C_k = \int_\T \frac{f(T(z))}{(1 - \ov{s_k} z) i z} dz = \sum_{j=1}^n \int_\T \frac{f(\xi) r_j'(\xi)}{(1 - \ov{s_k} r_j(\xi)) i r_j(\xi)} d\xi\]
            \[= \frac1i \int_\T f(\xi) \sum_{j=1}^n \l(\frac{\ov{s_k}}{1 - \ov{s_k}r_j(\xi)} + \frac1{r_j(\xi)}\r)r_j'(\xi) d\xi\]
            \[= \frac1i \int_\T f(\xi) d\Bigg(-\ln\Bigg(\prod_{j=1}^n [1 / \ov{s_k} - r_j(\xi)]\Bigg) + \ln\Bigg(\prod_{j=1}^n r_j(\xi)\Bigg)\Bigg).\]
        \end{samepage}
        Now it is  sufficient to prove that the logarithmic derivative of $\prod_{j=1}^n \l[1 / \ov{s_k} - r_j(\xi)\r]$ with respect to $\xi$ does not depend on $k$. Indeed, notice that $r_1(\xi) .. r_n(\xi)$ are precisely all the roots of the monic polynomial
        \[p(z) = \frac1{1 - \xi \prod(-\ov{s_j})} \Bigg(\prod_{j=1}^n (z - s_j) - \xi \prod_{j=1}^n (1 - \ov{s_j} z)\Bigg).\]
       Hence,
        \[\prod_{j=1}^n \l[1 / \ov{s_k} - r_j(\xi)\r] = p(1/\ov{s_k}) = \frac{\prod_{j=1}^n (1/\ov{s_k} - s_j)}{1 - \xi \prod_{j=1}^{n}(-\ov{s_j})}\]
        which depends on $k$ purely by a constant factor.
    \end{proof}

    \begin{theorem}
        \label{circmin}
        Let $\a \in \C^n, |\a_k| > 1, k = 1..n$. Let $\Phi : \R \to \R$ be any convex even function. Let
        \[L_n(\a,\b) = \int_\T \Phi(|P_n(\a, \b, z)|) |dz|,\]
        where $\b \in \C^{n-1}$ and
        \eq{P_n(\a,\b,z) = \frac{\prod_{k=1}^{n-1} (z - \b_k)}{\prod_{k=1}^n (z - \a_k)}.}P
        Then,
        \[\min_{\b \in \C^{n-1}} L_n(\a,\b) = \int_\T \Phi\Bigg(\frac1{|z - \prod_{k=1}^n|\a_k||}\Bigg)|dz|.\]
    \end{theorem}
    \begin{proof}
        Because $|\b_k - z| = |z\ov{\b_k} - 1|$ for $z \in \T$, we can replce $P_n$ in $L_n$ with
        \eq{Q(\a,\b,z) = \frac{\prod_{k=1}^{n-1}(z\ov{\b_k} - 1)}{\prod_{k=1}^n(z - \a_k)},}Q
        that is $|Q(\a,\b,z)| = |P_n(\a,\b,z)|$,  $z \in \T$.

        We use the following identity:
        \eq{\frac1{z - p} = -\frac1{|p|^2 - 1}\l(\ov p + \frac{\ov p}p \mu(1/\ov p, z)\r),}{recshift}
        where $p$ is any complex constant and
        \[\mu(c,z) = \frac{z - c}{1 - \ov c z}\]
        is the Möbius transform (for $|c| < 1$). The identity may be checked by direct computation. 
        It is also true that $\mu(c, \mu(-c, z)) = z$.

        For a fixed $\a\in\mathbb C^n$, we set
        \eq{Q^0(\a,z) = -\frac1{\prod|\a_k|^2 - 1}\l(\prod\ov{\a_k} - \prod\frac{\ov{\a_k}}{\a_k} \cdot R^*(\a,z)\r),}{Q*}
        where
        \[R^*(\a,z) = \prod_{k=1}^{n} \frac{z - 1/\ov{\a_k}}{z/\a_k - 1}.\]

        Expanding (\ref{Q*}) we can show that $Q^0$ has the same form as $Q$. Indeed, from
        \[Q^0(\a,z) = \frac{\prod_{k=1}^{n} (z \ov \a_k - 1) - \prod_{k=1}^{n} \ov \a_k (z - \a_k)}{(\prod_{k=1}^{n}|\a_k|^2 - 1)\prod_{k=1}^{n} (z - \a_k)},\]
        we  see that the polynomial
        \[q(z) = \frac{\prod_{k=1}^{n} (z \ov \a_k - 1) - \prod_{k=1}^{n} \ov \a_k (z - \a_k)}{\prod_{k=1}^{n}|\a_k|^2 - 1}\]
        is of degree at most $n-1$ and its constant term equals $(-1)^{n-1}$. Therefore it can be expressed as $\prod_{k=1}^{n-1}\big(\ov{\t\b_k}z - 1\big)$ with some values $\t\b_k\in\mathbb C$ depending on $\a,$ and we have $$Q^0(\a,z) = Q(\a,\t\b,z).$$ Moreover, $q(\a_k) \neq 0, k = 1..n$, thus $Q(\a, \t\b, \cdot)$, as defined in (\ref Q), is an irreducible fraction and $\{1/\ov{\t\b_k} | \t\b_k \neq 0\}$ are precisely the zeros of $Q^0(\a, \cdot)$, taking into account the multiplicities.
        
        Let us prove that 
        \[\min_{\b \in \C^{n-1}} L_n(\a,\b) =L_n(\a, \t\b) .\] 
         
        First, we show that
        \[ L_n(\a,\t\b) = \int_\T \Phi\Bigg(\frac1{|z - \prod_{k=1}^n|\a_k||}\Bigg)|dz|.\] 

        Define
        \eq{R^0(\a,z) = \mu\big(-\prod_{k=1}^{n} 1/\ov{\a_k}, -R^*(\a, z)\big).}{R0}
        Then, 
        \[R^*(\a,z) = -\mu\big(\prod_{k=1}^{n} 1/\ov{\a_k}, R^0(\a, z)\big).\]
        Substituting this into (\ref{Q*}) we obtain the expression of a form similar to the one on the right-hand side of (\ref{recshift}). So, from  (\ref{recshift}) we obtain
        \eq{Q^0(\a,z) = \frac1{R^0(\a, z) - \prod \a_k}.}{Q0}

        Because $\mu(c, \cdot)$ is of degree $1$ and $R^*$ is of degree $n,$ we  conclude from (\ref{R0}) that $R^0$ must also be of degree $n$. Plugging $0$ into the definition of $R^0$ it is also easy to see
         that the latter has a root at origin. Hence, the representation
        \[R^0(\a,z) = W z \prod_{k=1}^{n-1} \frac{z - \h\b_k}{\ov{\h\b_k}z - 1}\]
        holds for some values $|\h\b_k| < 1$ and $|W| = 1$. Substituting it into (\ref{Q0}) we get
        \[Q^0(\a, z) = \frac{\prod\big(\ov{\h\b_k} z - 1\big)}{W z \prod \big(z - \ov{\h\b_k}\big) - \prod\big(\ov{\h\b_k} z - 1\big) \cdot \prod \a_k}.\]
        We can see that $\{1/\ov{\h\b_k} | \h\b_k \neq 0\}$ are precisely all zeros of $Q^0(\a, \cdot)$ accounting for multiplicity, much like for $\t\b$. So we can write $\h\b = \t\b$. Comparing the numerator's constant term and the denominator's leading coefficient to those of (\ref Q), we can also conclude that $W = 1$. Therefore, we have
        \[Q^0(\a,z) = \Bigg(z\prod_{k=1}^{n-1}\frac{z - \t\b_k}{\ov{\t\b_k}z - 1} - \prod_{k=1}^n \a_k\Bigg)^{-1}.\]

        Because $R^0$ is a Blaschke product with a root at the origin, for any continuous \mbox{$g: \T \to \R$}, we have
        \[\int_\T g(R^0(\a, z)) |dz| = \int_\T g(z) |dz|.\]
        Indeed, define $g^*$ to be the harmonic continuation of $g$ into $\D$. We see that $g(R^0(\a, \cdot))$ is also harmonic and 
        \[\int_\T g(R^0(\a, z)) |dz| = 2\pi g^*(R^0(\a, 0)) = 2\pi g^*(0) = \int_\T g(z) |dz|.\]
        It means that we can replace $R^0$ under the integral defining $L_n(\a, \t\b)$ with  $z$ and get
        \[L_n(\a,\t\b) = \int_\T\Phi\Bigg(\frac1{|R^0(\a, \t\b, z) - \prod \a_k|}\Bigg) |dz| = \int_\T\Phi\Bigg(\frac1{|z - \prod|\alpha_k||}\Bigg)|dz|.\]

        Now it's left to prove that this is indeed the minimum of $L_n(\a, \b)$ over all $\b\in\mathbb C^{n-1}$ for a fixed $\a \in\mathbb C^n.$

        Without loss of generality we assume that $\Phi$ is continuously differentiable. Otherwise, we uniformly approximate $\Phi$ by a sequence of even convex functions $\Phi_m \in C^1(\R)$. WLOG, we can also assume that the poles $\a_k$ are distinct. Otherwise, we approximate given configuration of poles by another one consisting of pairwise distinct points and  use the continuity of the maps $\a \to \t\b$ and $\a, \b \to L_n(\a, \b)$.

        Decompose $Q(\a, \b, z)$ into partial fractions:
        \[Q(\a, \b, z) = \sum_{k=1}^n \frac{\omega_k}{z - \a_k}.\]
        Compare the constant term in the numerator in (\ref{Q}) with the same one in the fraction obtained by recombining the sum:
        \[(-1)^{n-1} = (-1)^{n-1} \sum_{k=1}^n \frac{\omega_k}{\a_k} \prod_{j=1}^n \a_j.\]
        Let $v_k = \omega_k/\a_k \prod \a_k$. We have:
        \[Q(\a,\b,z) = \sum_{k=1}^n v_k \frac{\a_k}{z - \a_k} \cdot \prod \frac1{\a_j},\]
        \eq{\sum_{k=1}^n v_k = 1.}{vsum}
        The condition (\ref{vsum}) defines the space of all possible values for the vector $v = \{v_k\}_{k=1}^n$, and
        there is a bijection between the space $\mathbb C^{n-1}$ of values of $\b$
        and the space of all $v = \{v_k\}_{k=1}^n\in\mathbb C^n$ satisfying  (\ref{vsum}).
        This allows us to speak about minimality of $L_n$ expressed in terms of $v$ rather that $\b$. Let $Q_v(\a, v, z) = Q(\a, \b, z)$ be the expression of $Q$ in terms of $v$. $|Q_v|$ is a convex function of $v$ and $\Phi(|Q_v|)$ is a convex function of $v$. Thus $L_n$ is also convex in terms of $v$, and it is sufficient to show that $\t v$ is a stationary point, where  $ v=\t v$ corresponds to $\b = \t\b$. 

        Let $dv$ be a variation of $v$ at $\t v$. Then
        \[dL = \sum_{k=1}^n \l(dv_k \int_\T \frac{\p \Phi(|Q_v|)}{\p v_k}|dz| + c.c.\r)\]
        \[= \sum_{k=1}^n \l(dv_k \prod \frac1{\a_j} \int_\T \Phi'(|Q_v|) \frac{\ov{Q_v}}{2 |Q_v|} \frac{\a_k}{z - \a_k} |dz| + c.c.\r)\]
        \[= \sum_{k=1}^n \Bigg(dv_k \prod \frac1{\a_j} \int_\T \frac{H(Q_v)}{\ov{1/\ov{\a_k}}z - 1} |dz| + c.c.\Bigg),\]
        where $c.c.$ stands for complex conjugate of the first summand and
        \[H(Q_v) = \Phi'(Q_v)\frac{\ov{Q_v}}{2|Q_v|}.\]
        Because $Q_v(\a, \t v, \cdot)$ can be expressed as a function of $R^*(\a, z)$, which is a Blaschke product with zeros at $1/\ov{\a_k}$ (up to a unimodular factor), we can apply Lemma \ref{Clema} and get
        \[\exists C : \int_\T \frac{H(Q_v)}{\ov{1/\ov{\a_k}}z - 1} |dz| = C\ \forall k = 1..n.\]
        From (\ref{vsum}) we  see that $\sum dv_k = 0$. Hence, $dL = 0$.

    \end{proof}

    \section{Proof of Theorem 1}

    We use notations  $a = (a_k)_{k=1}^n$, $b = (b_k)_{k=1}^{n-1}$. Express $\I_n$ in terms of circular integrals:
    \[\I_n(a, b) = \int_0^1 \rho(r) L_n^r(a, b) dr,\]
    where \[L_n^r(a, b) = \int_{|z|=r} \Psi\l(\l| \frac{\prod_{k=1}^{n-1} (z - b_k)}{\prod_{k=1}^n (z - a_k)} \r|\r) |dz|.\]
    Obviously, if (\ref{condb}) is the minimality condition for $L_n^r\ \forall r \in (0, 1)$, then it is the minimality condition for $\I_n$.

    For a fixed value of $r$, we set $\a_k = a_k/r$, $\b_k = b_k/r$ and $\Phi(x) = \Psi(x/r)$. Then,
    \[L_n^r(a, b) = \int_\T \Psi\l(\l|\frac{\prod_{k=1}^{n-1}(rz - b_k)}{\prod_{k=1}^n(rz - a_k)}\r|\r) r |dz| = r \int_\T \Phi\l(|P_n(\a, \b, z)|\r) |dz|,\]
    where $P_n$ is defined in (\ref P).
    Plugging (\ref{condb}) into this, we get
    \[\frac{L_n^r(a, b)}r = \int_\T \Phi\l(\l|\frac{z^{n-1}}{z^n + (-r)^{-n}}\r|\r) |dz| = \int_\T \Phi\l(\frac1{|z - r^{-n}|}\r) |dz|,\]
    which is the minimal value according to Theorem \ref{circmin}. This finishes the proof of Theorem \ref{tgen}.

    \section*{Acknowledgments}

    The author is grateful to Natalia Abuzyarova for contributing to the text and helpful editorial revisions.

    \printbibliography
\end{document}